\RequirePackage{etex}
\RequirePackage{easymat}

\documentclass[12pt]{elsarticle}
\usepackage{amsmath,amsthm,bbm,amssymb,extsizes}

\usepackage[all]{xy}
\usepackage[active]{srcltx}

\DeclareMathOperator{\Ker}{Ker}
\DeclareMathOperator{\diag}{diag}
\renewcommand{\le}{\leqslant}
\renewcommand{\ge}{\geqslant}

\newtheorem{theorem}{Theorem}
\newtheorem*{lemma}{Lemma}
\newtheorem*{corollary}{Corollary of Theorems \ref{the}
and  \ref{jtt}}

\theoremstyle{definition}
\newtheorem*{definition}{Definition}

\begin{document}
\title{Topological classification of chains of
linear mappings}

\author[ryb]{Tetiana Rybalkina}
\ead{rybalkina\_t@ukr.net}

\author[ryb]{Vladimir V. Sergeichuk\corref{cor}}
\ead{sergeich@imath.kiev.ua}
\address[ryb]{Institute of Mathematics, Tereshchenkivska 3,
Kiev, Ukraine}
\cortext[cor]{Corresponding author}

\begin{abstract}
We consider systems of linear mappings
${A}_1,\dots,{A}_{t-1}$ of the form
\[{\cal A}:\quad
\xymatrix{
U_1\ar@{-}^{{A}_1}[r]
 &
U_2 \ar@{-}^{{A}_2}[r]
 &
U_3 \ar@{-}^{{A}_3}[r]
 &
\cdots \ar@{-}^{{A}_{t-1}}[r]
  &
U_{t}
}
\]
in which $U_1,\dots,U_t$ are unitary
(or Euclidean) spaces and each line is
either the arrow $\longrightarrow$ or
the arrow $\longleftarrow$. Let ${\cal
A}$ be transformed to
\[{\cal B}:\quad
\xymatrix{
V_1\ar@{-}^{{B}_1}[r]
 &
V_2 \ar@{-}^{{B}_2}[r]
 &
V_3 \ar@{-}^{{B}_3}[r]
 &
\cdots \ar@{-}^{{B}_{t-1}}[r]
  &
V_{t}
}
\]
by a system $\{\varphi_i:U_i\to
V_i\}_{i=1}^t$ of bijections. We say
that ${\cal A}$ and ${\cal B}$ are
linearly isomorphic if all $\varphi_i$
are linear. Considering all $U_i$ and
$V_i$ as metric spaces, we say that
${\cal A}$ and ${\cal B}$ are
topologically isomorphic if all
$\varphi_i$ and $\varphi_i^{-1}$ are
continuous.

We prove that ${\cal A}$ and ${\cal B}$
are topologically isomorphic if and
only if they are linearly isomorphic.
\end{abstract}

\begin{keyword}
Chains of linear mappings\sep
Topological equivalence

\MSC 5A21; 37C15
\end{keyword}

\maketitle

\section{Introduction and theorem}
A \emph{chain of linear mappings} is a
system of linear mappings
${{A}_1},\dots,{{A}_{t-1}}$ of the form
\begin{equation}\label{dfy}
{\cal A}:\quad
\xymatrix{
U_1\ar@{-}^{{A}_1}[r]
 &
U_2 \ar@{-}^{{A}_2}[r]
 &
U_3 \ar@{-}^{{A}_3}[r]
 &
\cdots \ar@{-}^{{A}_{t-1}}[r]
  &
U_{t}
}
\end{equation}
in which each line is either the arrow
$\longrightarrow$ or the arrow
$\longleftarrow$. We assume that
$U_1,\dots,U_t$ are unitary spaces (or
are Euclidean spaces). Without loss of generality, the reader may think that all $U_1,\dots,U_t$ are $\mathbb C\oplus \dots\oplus \mathbb C$ (or $\mathbb R\oplus \dots\oplus \mathbb R$, respectively) with a natural topology on them.

Let
\begin{equation}\label{dfy1}
{\cal B}:\quad
\xymatrix{
V_1\ar@{-}^{{B}_1}[r]
 &
V_2 \ar@{-}^{{B}_2}[r]
 &
V_3 \ar@{-}^{{B}_3}[r]
 &
\cdots \ar@{-}^{{B}_{t-1}}[r]
  &
V_{t}
}\end{equation}
be a chain with the same orientation of arrows as in \eqref{dfy}. We write $\varphi:\cal
A\xrightarrow{\sim}{\cal B}$ if $\varphi=\{\varphi_i:U_i\to
V_i\}_{i=1}^t$ is a system of bijections such that all
squares in the diagram
    \begin{equation*}\label{rwu}
\xymatrix{
U_1\ar@{-}^{{A}_1}[r]\ar[d]_{\varphi_1}
 &
U_2 \ar@{-}^{{A}_2}[r]\ar[d]_{\varphi_2}
 &
U_3 \ar@{-}^{{A}_3}[r]\ar[d]_{\varphi_3}
 &
\cdots \ar@{-}^{{A}_{t-2}}[r]
  &
U_{t-1} \ar@{-}^{{A}_{t-1}}[r]
\ar[d]_{\varphi_{t-1}}
  &
U_t\ar[d]_{\varphi_t}
           \\
V_1\ar@{-}^{{B}_1}[r]
 &
V_2 \ar@{-}^{{B}_2}[r]
 &
V_3 \ar@{-}^{{B}_3}[r]
 &
\cdots \ar@{-}^{{B}_{t-2}}[r]
  &
V_{t-1} \ar@{-}^{{B}_{t-1}}[r]
  &
V_t
 }
\end{equation*}
are commutative; that is,
\begin{align*}
\varphi_{i+1}{A}_i={B}_i\varphi_i\quad
\text{if }{A}_i:U_i\to U_{i+1}\\
\varphi_{i}{A}_i= {B}_i\varphi_{i+1}\quad
\text{if }
{A}_i:U_i\leftarrow U_{i+1}
\end{align*}
for each $i=1,\dots,t-1$.

\begin{definition}\label{def}
We say that $\varphi:\cal
A\xrightarrow{\sim}{\cal B}$ is \begin{itemize}
  \item[\rm(i)] an \emph{isometry}
      if each $\varphi_i:U_i\to
      V_i$ is a linear bijection
      that preserves the scalar
product; that is, each $\varphi_i$
is a unitary map (or an orthogonal
map if all spaces are Euclidean);

  \item[\rm(ii)] a \emph{linear
      isomorphism} if each
      $\varphi_i:U_i\to V_i$ is a
      linear bijection (in this
      definition, we forget that
      $U_i$ and $V_i$ are metric
      spaces and consider them as
      linear spaces);

  \item[\rm(iii)] a
      \emph{topological
      isomorphism} if each
      $\varphi_i:U_i\to V_i$ is a
      homeomorphism, which means
that $\varphi_i$ and
$\varphi_i^{-1}$ are continuous and
bijective (we forget that $U_i$ and
$V_i$ are
      linear spaces and consider
      them as metric spaces).
\end{itemize}
\end{definition}

Each linear bijection of unitary (or
Euclidean) spaces is a homeomorphism,
hence
\begin{align*}
  &\text{$\varphi:\cal
A\xrightarrow{\sim}{\cal B}$ is
an
isometry}\\&\qquad\qquad\Longrightarrow\text{$\varphi$
is a linear
isomorphism}\\&\qquad\qquad\qquad\qquad\Longrightarrow\text{$\varphi$
is a topological isomorphism.}
\end{align*}

The main result of this paper is the
following theorem, which is proved in
Section \ref{s4}.

\begin{theorem} \label{the}
Two chains of linear mappings on
unitary (or Euclidean) spaces are
topologically isomorphic if and only if
they are linearly isomorphic.
\end{theorem}

Note that the problem of topological
classification was also studied for
linear operators \cite{Capp-conexamp,
Capp-2th-nas-n<=6, Capp-big-n<6,
Pardon,Kuip-Robb,bud1} (Budnitska is the maiden name of the first author), affine
operators
\cite{Blanc,Ephr,bud,bud1,bud+bud},
dynamical systems \cite{Robb}, and
representations of Lie groups
\cite{Schultz}.

The paper is organized as follows. In  Section \ref{s2} we show that the problem of classifying chains
\eqref{dfy} up to isometry is hopeless for each $t \ge 3$.
In  Section \ref{s3} we recall a known classification of chains
\eqref{dfy} up to linear isomorphism; we formulate it in terms of dimensions of some subspaces. In  Section \ref{s4} we show that these dimensions are also topological invariants, which proves Theorem \ref{the}.

\section{Isometry of chains}\label{s2}

In this section, we consider chains
\eqref{dfy} of linear mappings on
\emph{unitary} spaces. It would be the
most natural to classify them up to
isometry.

If $t=2$, then the classification of
chains \eqref{dfy} up to isometry is
given by the singular value
decomposition: there exist orthonormal
bases in $U_1$ and $U_2$ in which the
matrix of ${A}_1$ is $\diag(a_1,\dots,
a_r)\oplus 0$, where $a_1\ge\cdots\ge
a_r>0$ are real numbers that are
uniquely determined by ${A}_1$.

Unfortunately, \emph{the problem of
classifying chains up to isometry must
be considered as hopeless for $t=3$
(and so for each $t \ge 3$)} since it
contains the problem of classifying
linear operators on unitary
spaces up to unitary similarity, and hence all
systems of linear mappings on unitary
spaces (see the end of this section). This statement is proved sketchy in
\cite[Section 2.3]{ser-uni}; for the
reader convenience we prove in detail the following
weaker assertion.

\begin{theorem} \label{thee}
The problem of classifying chains
\begin{equation}\label{o5ds}
U_1\stackrel{{A}_1}
{\longrightarrow}U_2 \stackrel{{A}_2}
{\longleftarrow}U_3,\qquad\qquad U_1,U_2,U_3
\text{ are unitary spaces,}
\end{equation}
up to isometry contains the problem of
classifying linear operators on unitary
spaces up to unitary similarity.
\end{theorem}

\begin{proof}
We say that matrices $X$ and $Y$ are
\emph{unitarily similar} if there
exists a unitary matrix $S$ such that
$S^{-1}XS=Y$.

Let us consider chains of mappings
\begin{equation}\label{pwy}
U_1\stackrel{{A}_1}
{\longrightarrow}U_2 \stackrel{{A}_2}
{\longleftarrow}U_3\quad\text{and}\quad
V_1\stackrel{{B}_1}
{\longrightarrow}V_2 \stackrel{{B}_2}
{\longleftarrow}V_3
\end{equation}
that are given in some orthonormal
bases by pairs of matrices $(M,N_X)$
and $(M,N_Y)$ in which
\[
M:=
\begin{bmatrix}
        I&0&0 \\ 0&2I&0 \\ 0&0&3I
\end{bmatrix},\quad N_X:=\begin{bmatrix}
        I&0\\ I&I\\ I&X
\end{bmatrix},\quad N_Y:= \begin{bmatrix}
        I&0\\ I&I\\ I&Y
\end{bmatrix},
\]
and all blocks are $m\times m$.

It suffices to prove that
\begin{equation}\label{48b}
\parbox{25em}
{the chains \eqref{pwy} are isometric if
and only if $X$ and $Y$ are unitarily
similar.}
\end{equation}
Indeed, assume we know a set of
canonical matrix pairs for
\eqref{o5ds}. We take those of them
that can be reduced to the form $(M,N_X)$
and reduce them to it. Due to
\eqref{48b}, the obtained blocks $X$
form a set of canonical matrices for
unitary similarity.

Let us prove \eqref{48b}.

``$\Longrightarrow$'' \ Let the chains
\eqref{pwy} be isometric; that is,
there exist unitary matrices
$S_1,S_2,S_3$ such that
\begin{equation}\label{jups}
S_2^{-1}MS_1=M,\qquad S_2^{-1}N_XS_3=N_Y.
\end{equation}
By the first equality in \eqref{jups},
\[S_1^*M^*S_2=M^*,\qquad
MM^*S_2=MS_1M^*=S_2MM^*.\] Since
$MM^*=I_m\oplus 4I_m \oplus 9I_m$, we have
$S_2=C_1\oplus C_2\oplus C_3$ for some
$m\times m$ matrices $C_1,C_2$, and
$C_3$.

By the second equality in \eqref{jups},
$S_2N_Y=N_XS_3$. Equating the
corresponding horizontal strips, we
obtain
\begin{equation}\label{g9o}
\begin{bmatrix}C_1& 0\end{bmatrix}
=\begin{bmatrix}I& 0\end{bmatrix}S_3,
\
\begin{bmatrix}C_2& C_2\end{bmatrix}
=\begin{bmatrix}I& I\end{bmatrix}S_3,\
\begin{bmatrix}C_3& C_3Y\end{bmatrix}
=\begin{bmatrix}I& X\end{bmatrix}S_3.
\end{equation}
Let $S_3=[R_{ij}]_{i,j=1}^2$. The first
equality in \eqref{g9o} implies that
$R_{11}=C_1$ and $R_{12}=0$. Since
$S_3$ is unitary, $R_{21}=0$ and so
$S_3=C_1\oplus R_{22}$. The second
equality in \eqref{g9o} implies that
$C_1=C_2=R_{22}$. The third equality in
\eqref{g9o} implies that
$C_3=C_2=R_{22}$ and $C_3Y=XC_3$. Thus,
$X$ and $Y$ are unitarily similar.

``$\Longleftarrow$'' \ Conversely, if
$C^{-1}XC=Y$ for some unitary $C$, then
\eqref{jups} holds for $S_1=S_2=C\oplus
C\oplus C$ and $S_3=C\oplus C$, and so
the chains \eqref{pwy} are isometric.
\end{proof}

Recall that a \emph{quiver} is a directed graph. Its \emph{representation} is given by assigning to each vertex a unitary space and to each arrow a linear mapping of the corresponding vector spaces. A representation is \emph{unitary} if all of its vector spaces are unitary.

It was shown in \cite[Section 2.3]{ser-uni} that the problem of classifying linear operators on unitary spaces up to unitary similarity
contains the problem of classifying unitary representations of an \emph{arbitrary} quiver. Thus, we cannot expect to find an observable system of invariants for linear operators on unitary
spaces. Nevertheless, we can reduce the matrix of \emph{any given} linear operator on a unitary space (moreover, the matrices of {any given} unitary representation of a quiver) to canonical form by using  Littlewood's algorithm; see \cite[Section 3]{ser-uni}.

In the same way, the problem of classifying pairs of linear operators on a vector space is considered as hopeless (and all classification problems that contain it are called \emph{wild}) since it contains the problem of classifying representations of each quiver. Nevertheless, we can reduce the matrices of any given representation of a quiver to canonical form by using Belitskii's algorithm; see \cite{bel-ser,ser}.

\section{Linear isomorphism of chains}
\label{s3}

In this section, we consider chains of
linear mappings
\begin{equation}\label{dfys}
{\cal A}:\quad
\xymatrix{
U_1\ar@{-}^{{A}_1}[r]
 &
U_2 \ar@{-}^{{A}_2}[r]
 &
U_3 \ar@{-}^{{A}_3}[r]
 &
\cdots \ar@{-}^{{A}_{t-1}}[r]
  &
U_{t}
}
\end{equation}
on vector spaces without scalar
product. Without complicating the proofs, we consider them over \emph{any} field $\mathbb
F$. In Theorem \ref{jtt} we recall the
well-known classification of such
chains up to linear isomorphisms (see
Definition \ref{def}(ii)). Next we fix
some subspaces of $U_1,\dots,U_t$ and
prove in Theorem \ref{ljtt} that the
set of their dimensions is a full
system of invariants of chains with
respect to linear isomorphisms. In
Section \ref{s4} we establish that this
set is also a full system of invariants
of chains with respect to topological
isomorphisms, which proves Theorem
\ref{the}.

\subsection{A classification of chains up
to linear isomorphisms}

The directions ($U_i\to U_{i+1}$ or
$U_i\leftarrow U_{i+1}$) of all linear
mappings ${A}_i$ in \eqref{dfys} can be
given by the directed graph
\begin{equation}\label{kuea}
G:\quad
\xymatrix{
1\ar@{-}^{\alpha_1}[r]
 &
2 \ar@{-}^{\alpha_2}[r]
 &
3 \ar@{-}^{\alpha_3}[r]
 &
\cdots \ar@{-}^{\alpha_{t-1}}[r]
  &
{t}
}
\end{equation}
in which each arrow $\alpha_i$ is
directed as ${A}_i$. (Thus, each chain
\eqref{dfys} defines a representation
of the quiver \eqref{kuea} and a linear
isomorphism of chains defines an
isomorphism of the corresponding
representations.) Write
\begin{equation}\label{fsk}
{\cal A}(k):=U_k,\qquad k=1,\dots, t.
\end{equation}
The \emph{direct sum} of chains ${\cal
A}$ and
\[
{\cal B}:\quad
\xymatrix{
V_1\ar@{-}^{{B}_1}[r]
 &
V_2 \ar@{-}^{{B}_2}[r]
 &
V_3 \ar@{-}^{{B}_3}[r]
 &
\cdots \ar@{-}^{{B}_{t-1}}[r]
  &
V_{t}
}\]
 with the same directed graph
\eqref{kuea} is the chain
\begin{equation*}\label{reph}
{\cal A}\oplus{\cal B}:\quad U_1\oplus V_1\,
\frac{\ \scriptstyle{A}_1\oplus{B}_1\ }{\qquad}
\,U_2\oplus V_2\,
\frac{\ \scriptstyle{A}_2\oplus{B}_2\ }{}\,\cdots\,
\frac{\ \scriptstyle{A}_{t-1}\oplus{B}_{t-1}\ }{\qquad}\,
U_{t}\oplus V_t.
\end{equation*}

For every pair of integers $(i,j)$ such
that $1\le i\le j\le t$, we define the
chain
\begin{equation*}\label{gte1}
{\cal L}_{ij}:\quad 0 \,\frac{\ \ \ }{}\,\cdots\,
\frac{\ \ \ }{}\,0\, \frac{\ \ \ }{}\, \mathbb F
\,\frac{\ \scriptstyle\mathbbm{1}\ }{}\,
\mathbb F \,\frac{\ \scriptstyle\mathbbm{1} \ }{}\,
\cdots \, \frac{\ \scriptstyle\mathbbm{1} \ }{}\, \mathbb F\,
\frac{\ \ \ }{}\,0\,
\frac{\ \ \ }{}\,\cdots\, \frac{\ \ \ }{}\,0
\end{equation*}
in which ``$\mathbbm{1}$'' is the
identity bijection and $\mathbb F$'s
are at the vertices $i,i+1,\dots,j$ of
\eqref{kuea}.

The following theorem is well known in
the theory of quiver representations;
the representations of \eqref{kuea} and
the other quivers that have a finite
number of nonisomorphic indecomposable
representations were classified by
Gabriel \cite{gab}.

\begin{theorem}\label{jtt}
Each chain $\cal A$ is linearly
isomorphic to a direct sum of chains of
the form ${\cal L}_{ij}$. This direct
sum is uniquely determined by $\cal A$,
up to permutation of summands.
\end{theorem}

An algorithm for constructing this
canonical form of chains of linear
mappings over $\mathbb C$ is given in
\cite[Section 4]{ser-cycle}; it uses
only transformations of unitary
equivalence of matrices: $M\mapsto
S_1MS_2$ in which $S_1$ and $S_2$ are
unitary.

\begin{corollary}
Each chain $\cal A$ of linear mappings
on unitary (or Euclidean) spaces is
topologically isomorphic to a direct
sum of chains of the form ${\cal
L}_{ij}$. This direct sum is uniquely
determined by $\cal A$, up to
permutation of summands.
\end{corollary}

Let $\cal A$ be any chain of the form
\eqref{dfys}. In each of its spaces
$U_i$, we define a series of subspaces
\begin{equation}\label{ryi}
0=U_{i0}\subset U_{i1}\subset U_{i2}\subset \cdots
\subset U_{ii} =U_{i},\qquad i=1,\dots,t
\end{equation}
by induction: $0=U_{10}\subset
U_{11}=U_1$ and if \eqref{ryi} is
constructed for $i<t$ then
\begin{multline}\label{kjy}
(U_{i+1,1},\dots,U_{i+1,i})\\:=
  \begin{cases}
({A}_iU_{i1},\dots, {A}_iU_{ii}) &
\text{if }{A}_i: U_i\to U_{i+1}, \\
(\Ker {A}_i, {A}_i^{-1}U_{i1}, \dots,
{A}_i^{-1}U_{i,i-1}) & \text{if }{A}_i:
U_i\leftarrow U_{i+1}
  \end{cases}
\end{multline}
(here $A_{i}^{-1}U_{ij}$ denotes the
preimage of $U_{ij}$).

\subsection{An example}
\label{ssss}

Each chain of the form
\begin{equation}\label{pgt}
{\cal A}:
\quad U_1\stackrel{{A}_1}
{\longrightarrow}U_2 \stackrel{{A}_2}
{\longleftarrow}U_3
\end{equation}
is given by the pair of matrices
$(M_1,M_2)$ in some bases of
$U_1,U_2,U_3$. Changing the bases, we
can reduce the pair by transformations
\begin{equation}\label{jup}
(M_1,M_2)\mapsto (S_2^{-1}M_1S_1,S_2^{-1}M_2S_3),
\qquad\text{$S_1,S_2,S_3$ are
nonsingular.}
\end{equation}
It is convenient to give $(M_1,M_2)$ by
the block matrix $[M_1|M_2]$ since the
rows of $M_1$ and $M_2$ are transformed
by the same matrix $S_2^{-1}$. Due to
\eqref{jup}, we can reduce it by
elementary row transformations (i.e.,
by simultaneous elementary
transformations with rows of $M_1$ and
$M_2$) and by elementary column
transformations within $M_1$ and $M_2$.
Each $[M_1|M_2]$ can be reduced by
these transformations to its canonical
form
\begin{equation}\label{wtu}
\begin{MAT}(c){|c|c|}
\first-
N_1&N_2\\-
\end{MAT}
=\begin{MAT}(c){|cc|c.c|}
\first-
0&I_p&
\begin{matrix}
0&I_r\\0&0
\end{matrix}&
\begin{matrix}
0\\0
\end{matrix}
\\.
0&0&
\begin{matrix}
0\\0
\end{matrix}&\begin{matrix}
I_q\\0
\end{matrix}\\-
\end{MAT}
\end{equation}
as follows (see \cite[Section
4]{ser-cycle}). We first reduce $M_1$
to the form
\begin{equation}\label{teo}
\begin{matrix}
  0 & I \\
  0 & 0 \\
\end{matrix}
\end{equation}
 and denote the obtained block
matrix by $[N_1|M'_2]$. Then we extend
to $M'_2$ the partition of $N_1$ into
two horizontal strips and reduce the
second horizontal strip of $M'_2$ to
the form \eqref{teo}:
\[
\begin{MAT}(c){|cc|c.c|}
\first-
0&I_p&
M_{11}&
M_{12}\\.
0&0&
\begin{matrix}
0\\0
\end{matrix}&\begin{matrix}
I_q\\0
\end{matrix}\\-
\end{MAT}
\]
We make $M_{12}$ equal to zero by
adding linear combinations of rows of
$I_q$. At last, we reduce $M_{11}$ to
the form \eqref{teo} by elementary
transformations; these transformations
may spoil $I_p$, we restore it by
column transformations. The obtained
block matrix has the form \eqref{wtu}.

For example, let the chain \eqref{pgt}
be given in some bases
$\{e_i\}_{i=1}^5$, $\{f_i\}_{i=1}^6$,
and $\{g_i\}_{i=1}^5$ of $U_1$, $U_2$,
and $U_3$ by the following canonical
block matrix of the form \eqref{wtu}:
\begin{equation*}\label{fwo}
\begin{MAT}(c){ccccccccccc}
&\scriptstyle\it e_1&\scriptstyle\it e_2&
\scriptstyle\it e_3&\scriptstyle\it e_4&
\scriptstyle\it e_5&\scriptstyle\it g_1&
\scriptstyle\it g_2&\scriptstyle\it g_3&
\scriptstyle\it g_4&\scriptstyle\it g_5\\
\scriptstyle\it f_1&0&0&1&0&0&   0&1&0&0&0\\
\scriptstyle\it f_2&0&0&0&1&0&   0&0&1&0&0\\
\scriptstyle\it f_3&0&0&0&0&1&   0&0&0&0&0\\
\scriptstyle\it f_4&0&0&0&0&0&   0&0&0&1&0\\
\scriptstyle\it f_5&0&0&0&0&0&   0&0&0&0&1\\
\scriptstyle\it f_6&0&0&0&0&0&   0&0&0&0&0
\addpath{(1,0,4)rrrrrrrrrruuuuuulllllllllldddddd}
\addpath{(6,0,4)uuuuuu}
\addpath{(1,3,.)rrrrrrrrrr}
\addpath{(9,0,.)uuuuuu}\\
\end{MAT}
\end{equation*}
Then $\cal A$ is the direct sum of 9
chains that are given by the action of
the mappings on the basic vectors as
follows:
\begin{equation*}\label{hro}
\begin{split}
\xymatrix@C=20mm@R=2mm{
e_1\ar@{|->}[r] &0&0 \ar@{|->}[l]\\
e_2\ar@{|->}[r] &0&0 \ar@{|->}[l]\\
0\ar@{|->}[r] &0&g_1 \ar@{|->}[l]\\
e_3\ar@{|->}[r] &f_1&g_2 \ar@{|->}[l]\\
e_4\ar@{|->}[r] &f_2&g_3 \ar@{|->}[l]\\
e_5\ar@{|->}[r] &f_3&0 \ar@{|->}[l]\\
0\ar@{|->}[r] &f_4&g_4 \ar@{|->}[l]\\
0\ar@{|->}[r] &f_5&g_5 \ar@{|->}[l]\\
0\ar@{|->}[r] &f_6&0 \ar@{|->}[l]\\
}\end{split}\end{equation*} (For
simplicity of notation, we write
$0\mapsto f_i$ instead of $0\mapsto
0$.) Thus, $\cal A$ is linearly
isomorphic to
\[
{\cal L}_{11}\oplus {\cal L}_{11}\oplus
{\cal L}_{33}\oplus {\cal L}_{13}\oplus
{\cal L}_{13}\oplus {\cal L}_{12}\oplus
{\cal L}_{23}\oplus {\cal L}_{23}\oplus
{\cal L}_{22}.
\]
The subspaces $U_{ij}$ defined in
\eqref{ryi} and \eqref{kjy} are the
following:
\begin{gather*}\nonumber
U_{10}=0,\quad U_{11}
=U_{1};\\\label{ryr} U_{20}=0,\quad
U_{21}=\langle f_1,f_2,f_3\rangle,\quad
U_{22}=U_2;\\\nonumber U_{30}=0,\quad
U_{31}=\langle g_1\rangle,\quad
U_{32}=\langle g_1,g_2,g_3\rangle,\quad
U_{33} =U_3;
\end{gather*}
here $\langle x,y,\dots, z\rangle$
denotes the subspace spanned by
$x,y,\dots, z$.

Note that
\begin{equation}\label{gey}
U_{32}=U_{31}\oplus\langle g_2\rangle
\oplus \langle g_3\rangle
\end{equation}
in which $\langle g_2\rangle$ and
$\langle g_3\rangle$ are the vector
spaces of the chains given by
\[\xymatrix{
e_3\ar@{|->}[r] &f_1&g_2 \ar@{|->}[l]\quad
\text{and}\quad
e_4\ar@{|->}[r] &f_2&g_3 \ar@{|->}[l]\\
}
\]

\subsection{A system of invariants}

\begin{theorem}\label{ljtt}
Each chain $\cal A$ is fully
determined, up to linear isomorphism,
by the indexed set
\begin{equation}\label{dsi}
\{n_{ij}\}_{1\le j\le i\le t}\qquad\text{in which }
n_{ij}:=\dim U_{ij}
\end{equation}
and $U_{ij}$  are defined in
\eqref{kjy}.
\end{theorem}

\begin{proof}
By Theorem \ref{jtt}, $\cal A$
possesses a canonical decomposition
\begin{equation}\label{gsik}
{\cal
A}:=\bigoplus_{\ell=1}^s{\cal A}_{\ell},\qquad
{\cal A}_{\ell}\simeq
{\cal L}_{p_{\ell}q_{\ell}},
\end{equation}
whose summands are determined up to
renumbering and linear isomorphisms of
summands. Thus, \emph{${\cal A}$ is
determined up to linear isomorphism by
the family of pairs
$\{(p_{\ell},q_{\ell})\}_{\ell=1}^s$}
and this family is determined by ${\cal
A}$ up to renumbering (i.e.,
$\{(p_{\ell},q_{\ell})\}_{\ell=1}^s$ is
an unordered set with repeating
elements).

For technical reason, it is better to
prove the following statements that are
stronger than the theorem:
\begin{itemize}
  \item[(i)] $\{n_{ij}\}_{1\le j\le
      i\le t}$ uniquely determines
      $\{(p_{\ell},q_{\ell})\}_{\ell=1}^{s}$,
      up to renumbering,
  \item[(ii)] there are indices
      $\ell (i,j)\in \{1,\dots,s\}$
      such that each of the spaces
      $U_{t1},\dots, U_{tt}$
      defined in \eqref{ryi} is
      decomposed into the direct
      sum
\begin{equation}\label{tp}
U_{ti}=U_{t,i-1}\oplus {\cal
A}_{\ell (i,1)}(t)\oplus\dots\oplus
{\cal A}_{\ell
(i,r_i)}(t)
\end{equation}
(see \eqref{fsk}; we put $r_i:=0$
if $U_{t,i-1}=U_{ti}$).\footnote{An
example of this decomposition is
given in \eqref{gey}, in which
$t=3$, $i=r_i=2$, ${\cal A}_{\ell
(2,1)}(3)=\langle g_2\rangle$, and
${\cal A}_{\ell (2,2)}(3)=\langle
g_3\rangle$.}
  \item[(iii)] all chains ${\cal
      A}_{\ell (i,1)},\dots, {\cal
      A}_{\ell (i,r_i)}$ have the
      first nonzero space at the
      same position, i.e.
\begin{equation*}\label{rte}
p_{\ell
(i,1)}=\dots= p_{\ell
(i,r_i)}=:a_i,
\end{equation*}
\item[(iv)] $a_i\ne a_j$ if $i\ne
    j$.
\end{itemize}

We use induction on $t$. The induction
base is trivial: the statements
(i)--(iv) hold for chains with $2$
vector spaces; that is, for
$U_1\stackrel{{A}_1}
{\longrightarrow}U_2$ and
$U_1\stackrel{{A}_1}
{\longleftarrow}U_2$.

Suppose that (i)--(iv) hold for chains
with $t-1$ vector spaces, in
particular, for the restriction
\[
{\cal A}':\quad
\xymatrix{
U_1\ar@{-}^{{A}_1}[r]
 &
U_2 \ar@{-}^{{A}_2}[r]
 &
U_3 \ar@{-}^{{A}_3}[r]
 &
\cdots \ar@{-}^{{A}_{t-2}}[r]
  &
U_{t-1}
}
\]
of ${\cal A}$ to the first $t-1$
spaces. We can suppose that the
summands in \eqref{gsik} are numbered
such that
\begin{equation}\label{sik}
\max(p_1,\dots,p_{s'}) <t= p_{s'+1}=\dots=p_s.
\end{equation}
The canonical decomposition of ${\cal
A}'$ can be obtained from \eqref{gsik}
as follows:
\begin{equation*}\label{gsie}
{\cal
A}':=\bigoplus_{\nu=1}^{s'}{\cal A}'_{\nu},
\qquad
{\cal A}'_{\nu}\simeq
{\cal L}_{p_{\nu}q'_{\nu}},\qquad
q_{\nu}':=\min(t-1,q_{\nu}),
\end{equation*}
in which $s'$ is defined in \eqref{sik}
and every ${\cal A}'_{\nu}$ is the
restriction of ${\cal A}_{\nu}$ to the
first $t-1$ vector spaces.

By induction hypothesis,
\begin{itemize}
  \item $\{n_{ij}\}_{1\le j\le i\le
      t-1}$ uniquely determines
      $\{(p_{\nu},q'_{\nu})\}_{\nu=1}^{s'}$,
      up to renumbering,
  \item there are indices $\nu
      (i,j)\in \{1,\dots,s'\}$ such
      that each of the spaces
      $U_{t-1,1},\dots,
      U_{t-1,t-1}$ is decomposed
      into the direct sum
\begin{align}\nonumber
U_{t-1,i}&=U_{t-1,i-1}\oplus {\cal
A}'_{\nu (i,1)}(t-1)\oplus\dots\oplus
{\cal A}'_{\nu
(i,r'_i)}(t-1)\\&=U_{t-1,i-1}\oplus {\cal
A}_{\nu (i,1)}(t-1)\oplus\dots\oplus
{\cal A}_{\nu
(i,r'_i)}(t-1),\label{tpje}
\end{align}
  \item $p_{\nu (i,1)}=\dots=
      p_{\nu (i,r'_i)}=:b_i,$

\item $b_i\ne b_j$ if $i\ne j$.
\end{itemize}

We suppose that the summands in
\eqref{gsik} are numbered such that
\begin{equation}\label{feo}
\begin{split}
{\cal
A}_{\nu (i,1)}(t)\ne 0,\ \dots,\ {\cal
A}_{\nu (i,k_i)}(t)\ne 0,\\ {\cal
A}_{\nu (i,k_i+1)}(t)=\dots={\cal
A}_{\nu (i,r'_i)}(t)=0.
\end{split}
\end{equation}

Let us prove (i)--(iv). Consider two
cases that are differ in the direction
of the last arrow in \eqref{kuea}.

\emph{Case 1:}
$\alpha_{t-1}:(t-1)\longrightarrow t$.
By \eqref{kjy},
\[
U_{t1}={A}_{t-1}U_{t-1,1},\ \dots,\
U_{t,t-1}={A}_{t-1}U_{t-1,t-1},\quad U_{tt}=U_t.
\]
By \eqref{tpje}, \eqref{feo}, and
\eqref{sik}, we have
\[
U_{ti}=
  \begin{cases}
   U_{t,i-1}\oplus {\cal
A}_{\nu (i,1)}(t)\oplus\dots\oplus
{\cal A}_{\nu
(i,k_i)}(t) & \text{if $i<t$},
               \\
 U_{t,t-1}\oplus {\cal
A}_{s'+1}(t)\oplus\dots\oplus
{\cal A}_{s}(t)    & \text{if $i=t$,}
  \end{cases}
\]
which is the desired decomposition
\eqref{tp}.

\emph{Case 2:}
$\alpha_{t-1}:(t-1)\longleftarrow t$.
By \eqref{kjy},
\[
U_{t1}=\Ker {A}_{t-1},\
U_{t2}={A}_{t-1}^{-1}U_{t-1,1},\ \dots,\
U_{tt}={A}_{t-1}^{-1}U_{t-1,t-1}=U_t.
\]
By \eqref{sik}, \eqref{tpje}, and
\eqref{feo}, we have
\[
U_{ti}=
  \begin{cases}
 {\cal
A}_{s'+1}(t)\oplus\dots\oplus
{\cal A}_{s}(t)    & \text{if $i=1$,}
\\
   U_{t,i-1}\oplus {\cal
A}_{\nu (i-1,1)}(t)\oplus\dots\oplus
{\cal A}_{\nu
(i-1,k_{i-1})}(t) & \text{if $i>1$},
  \end{cases}
\]
which is the desired decomposition
\eqref{tp}.

In both the cases, the family of pairs
$\{(p_{\nu},q_{\nu})\}_{\nu=1}^s$
(which is determined up to renumbering)
can be obtained from
$\{(p_{\nu},q'_{\nu})\}_{\nu=1}^{s'}$
by replacing $k_i$ pairs $(a_i,t-1)$
with $(a_i,t)$ for each $i=1,\dots,t-1$
and by attaching $k_t:=s-s'$ pairs
$(t,t)$. This proves the statement (i)
since $k_1,\dots,k_t$ are expressed via
$n_{ij}$:  \[k_i=\dim U_{ti}-\dim
U_{t,i-1}=n_{ti}-n_{t,i-1},\qquad
i=1,\dots,t\] (we set $n_{t0}:=0$).

The statements (ii)--(iv) follow from
the induction hypothesis and Cases 1
and 2.
\end{proof}

\section{Topological isomorphism of chains}
\label{s4}

The goal of this section is to prove
Theorem \ref{the}.

Let $\varphi:\cal
A\xrightarrow{\sim}{\cal B}$ be a
topological isomorphism of chains of
the form \eqref{dfy} and \eqref{dfy1}.
Due to Theorem \ref{ljtt}, it suffices
to prove that their sets \eqref{dsi}
coincide; that is, \[\dim U_{ij}=\dim
V_{ij}\qquad \text{for all }i,j,\] in
which $U_{ij}$ are the vector subspaces
of $U_i$ that were constructed in
\eqref{kjy}, and $V_{ij}$ are the
vector subspaces of $V_i$ that are
analogously constructed by the chain
${\cal B}$. Due to Definition
\ref{def}(iii), the topological
isomorphism $\varphi:{\cal
A}\xrightarrow{\sim}{\cal B}$ is formed
by the homeomorphisms $\varphi_i:U_i\to
V_i$. It suffices to show that each
$\varphi_i$ maps $U_{ij}$ on $V_{ij}$
since then each $U_{ij}$ is
homeomorphic to $V_{ij}$ and by
\cite{hur+wal} all homeomorphic vector
spaces have the same dimension. What is
left is to prove the following lemma.

\begin{lemma}\label{kiy}
If $\varphi:\cal
A\xrightarrow{\sim}{\cal B}$ is a
topological isomorphism of chains
\eqref{dfy} and \eqref{dfy1}, then
\begin{equation}\label{jte}
\varphi_iU_{ij}=V_{ij}\qquad\text{for all }
i=1,\dots,t\text{ and } j=1,\dots,i.
\end{equation}
\end{lemma}

\begin{proof} The assertion \eqref{jte} holds for
$i=1$ since $\varphi _1:U_1\to V_1$ is
a bijection. Suppose that \eqref{jte}
holds for $i=k$ (and all
$j=1,\dots,k$); let us prove it for
$i=k+1$. It suffices to prove that
\begin{equation}\label{jteh}
\varphi_{k+1}U_{k+1,j}\subset V_{k+1,j}
\qquad\text{for all }
j=1,\dots,k+1
\end{equation}
since then we can use \eqref{jteh} for
$\varphi^{-1}:{\cal
B}\xrightarrow{\sim} {\cal A}$ instead
of $\varphi $ and obtain
\[\varphi_{k+1}^{-1}V_{k+1,j}\subset
U_{k+1,j}
\qquad\text{for all }
j=1,\dots,k+1,\] which ensures
$\varphi_{k+1}U_{k+1,j}\supset
V_{k+1,j}$.

In the case ${A}_k:
      U_k\to U_{k+1}$, the
      inclusion \eqref{jteh} holds
      since if $y\in U_{k+1,j}$ and
      $x\in {A}_k^{-1}y\subset
U_{kj}$, then
\[
\xymatrix{
U_{kj}\ar@{->}^{{A}_k}_{\text{on}}[r]
\ar[d]_{\varphi_k}^{\text{biect}}
 &
U_{k+1,j} \ar[d]_{\varphi_{k+1}}
 &&
x \ar@{|->}[r]
\ar@{|->}[d]
  &
y\ar@{|->}[d]
           \\
V_{kj}\ar@{->}^{{B}_k}[r]
 &
V_{k+1}
 &&
\varphi_k x \ar@{|->}[r]
  &
{B}_k\varphi_k x
 }
\]
Thus, $\varphi_{k+1}y= {B}_k\varphi_k
x\in {B}_k V_{kj}=V_{k+1,j}$.

In the case ${A}_k:
      U_k\leftarrow U_{k+1}$,  the
      inclusion \eqref{jteh} holds
      since if $y\in U_{k+1,j}$
      then
\[
\xymatrix{
U_{kj}
\ar[d]_{\varphi_k}
 &
U_{k+1,j}
\ar@{->}_{{A}_k}[l]
 \ar[d]_{\varphi_{k+1}}
 &&
{A}_ky
\ar@{|->}[d]
  &
y\ar@{|->}[d]
\ar@{|->}[l]
           \\
V_{k}
 &
V_{k+1}\ar@{->}_{{B}_k}[l]
 &&
\varphi_k{A}_ky
  &
\varphi_{k+1}y
\ar@{|->}[l]
 }
\]
Thus, $ \varphi_{k}{A}_ky\in V_{kj}$
and so $\varphi_{k+1}y\in V_{k+1,j}$.
\end{proof}
\bigskip

The authors wish to express their thanks to the referee for several helpful comments.

\end{document}